\documentclass[english]{article}
\usepackage[T1]{fontenc}
\usepackage[latin9]{inputenc}
\usepackage{amsthm}
\usepackage{amsmath}
\usepackage{amssymb}
\usepackage{esint}

\makeatletter
\numberwithin{equation}{section}
  \theoremstyle{definition}
  \newtheorem{defn}{\protect\definitionname}[section]
  \theoremstyle{remark}
  \newtheorem{rem}{\protect\remarkname}[section]
  \theoremstyle{plain}
  \newtheorem{thm}{\protect\theoremname}[section]
  \theoremstyle{plain}
  \newtheorem{lem}{\protect\lemmaname}[section]
  \theoremstyle{plain}
  \newtheorem{prop}{\protect\propositionname}[section]

\makeatother

\usepackage{babel}
  \providecommand{\definitionname}{Definition}
  \providecommand{\lemmaname}{Lemma}
  \providecommand{\propositionname}{Proposition}
  \providecommand{\remarkname}{Remark}
\providecommand{\theoremname}{Theorem}

\begin{document}

\title{A Quasi-sure Non-degeneracy Property for the Brownian Rough Path}

\author{H. Boedihardjo%
\thanks{Department of Mathematics and Statistics , University of Reading,
Reading RG6 6AX, United Kingdom. Email: h.s.boedihardjo@reading.ac.uk.%
}, X. Geng%
\thanks{Mathematical Institute, University of Oxford, Oxford OX2 6GG, United
Kingdom. Email: xi.geng@maths.ox.ac.uk. %
}, X. Liu%
\thanks{Mathematical Institute, University of Oxford, Oxford OX2 6GG, United
Kingdom. Email: xuan.liu@maths.ox.ac.uk. %
} \ and Z. Qian%
\thanks{Mathematical Institute, University of Oxford, Oxford OX2 6GG, United
Kingdom. Email: zhongmin.qian@maths.ox.ac.uk.%
} }
\maketitle
\begin{abstract}
In the present paper, we are going to show that outside a slim set in the sense of Malliavin (or quasi-surely),
the signature path (which consists of iterated path integrals in every degree) of Brownian motion is non-self-intersecting.
This property relates closely to a non-degeneracy property for the Brownian rough path arising naturally from the uniqueness of signature
problem in rough path theory. As an important consequence we conclude that quasi-surely, the Brownian rough path does not have any tree-like pieces
and every sample path of Brownian motion is uniquely determined by its signature up to reparametrization.
\end{abstract}

\section{Introduction}

In 1954, motivated from the study of homotopy theory and loop space
homology, Chen \cite{Chen54} proposed a way of representing a vector-valued
path $x$ by a fully non-commutative tensor series 
\begin{equation}
S(x)=\sum_{n=0}^{\infty}\int_{0<t_{1}<\cdots<t_{n}<T}dx_{t_{1}}\otimes\cdots\otimes dx_{t_{n}}\label{eq: signature}
\end{equation}
of iterated path integrals. In recent literature, this representation
is known as the \textit{signature }of a path. Intuitively, the first
degree of $S(x)$ is the increment of $x,$ and the second degree
of $S(x)$ encodes the geometric signed area enclosed by $x$
and the chord connecting its end points. In general, the signature
is a global quantity which captures the total ``area'' in each degree produced by the
underlying path.

The fundamental importance of the signature representation lies in
the fact that it is essentially faithful: the signature uniquely determines
the underlying path in a certain sense. This is a deep point as it
reveals the relationship between local and global properties of a path.
The first result along this direction was contained in Chen's original
work \cite{Chen58} in 1958, in which he proved that an irreducible
and piecewise regular path is uniquely determined by its signature
up to translation and reparametrization. 

However, the class of paths Chen studied is very special as it does
not reveal a crucial invariance property of the signature map: a piece
along which the path $x$ goes out and traces back does not contribute
to the signature of $x$. The characterization of this invariance
property in a precise mathematical form is the key point of understanding
in what sense a generic path $x$ is uniquely determined by its signature.
It was after five decades that Hambly and Lyons \cite{HL10} first
gave a complete characterization in the case of continuous paths with
bounded variation. In particular, they showed that a continuous path
with bounded variation is uniquely determined by its signature up
to tree-like equivalence in their sense defined in terms of a height
function.

Since the work of Hambly and Lyons, many efforts have been made to
explore beyond the bounded variation setting. For applications in
probability theory, a natural class of paths to be considered is the
space of rough paths, as it is well known that a large amount of interesting
stochastic processes can be regarded as rough paths in a canonical
way. However, in the rough path setting, Hambly and Lyons' characterization
does not apply any more as their tree-like characterization forces
the underlying path to have bounded variation. It was in a recent
work of Boedihardjo, Geng, Lyons and Yang \cite{BGLY15} that the
right characterization for the above invariance property was identified
in terms of a real tree structure and the corresponding uniqueness
result for signature was established. 

On the other hand, if we consider the uniqueness problem for sample
paths of a stochastic process, we might expect stronger results since
a stochastic process usually has non-degenerate sample paths and the
above invariance phenomenon will not appear at all. A series of probabilistic
works have been done along this direction, originally for Brownian
motion by Le Jan and Qian \cite{LQ13}, which was later extended to
hypoelliptic diffusions by Geng and Qian \cite{GQ16} and Gaussian
processes by Boedihardjo and Geng \cite{BG15}. Formally the result
can be stated as the fact that with probability one, every sample
path of the underlying stochastic process is uniquely determined by
its signature up to translation and reparametrization.

The techniques involved in studying the uniqueness problem for signature
in the deterministic and probabilistic settings are very different.
Moreover, the deterministic result is weaker but it treats all possible
rough paths in one goal, while the probabilistic result is stronger
but we have to work in the support of the law of the underlying process
on path space. The link between the deterministic and probabilistic
approaches seems to be missing, and the main goal of the present paper
is to fill in this gap in a relatively robust way.

To be more precise, we will be interested in the following non-degeneracy
property for the Brownian rough path: \textit{it is not possible for
a path having a piece along which the path goes out and traces back}
(the precise mathematical statement will be made in the next section).
As discussed before, this non-degeneracy property arises naturally
from the deterministic uniqueness of signature problem for rough paths.
In particular, we are going to prove this non-degeneracy property
in Malliavin's capacity setting, which is much stronger than the probability
measure case and it reveals finer analytic structure over the Wiener
space than the underlying probability measure. According to the deterministic
uniqueness result for signature in \cite{BGLY15}, a direct consequence is that outside
a slim set in the sense of Malliavin or quasi-surely (see definition
in the next Section), every sample path of Brownian motion is uniquely determined
by its signature up to reparametrization. 

The main motivation of investigating quasi-sure analysis for the Brownian
rough path lies in the fundamental work of Sugita \cite{Sugita88}
in 1988 that capacity is a universal object with respect to a large
class of positive generalized Wiener functionals. Therefore, the quasi-sure
analysis provides a powerful universal tool in studying degenerate
functionals (for instance the Brownian bridge or pinned diffusions)
on the Brownian rough path. 

According to \cite{BGLY15}, the aforementioned quasi-sure non-degeneracy property for the Brownian rough path
is \textit{equivalent} to the quasi-sure non-self-intersection
for the Brownian signature path. Indeed, we are going to obtain a quantitative 
constraint on the degree $n$ of signature, the dimension $d$ of Brownian motion and the capacity
index $(r,q)$ (see (\ref{eq: capacity}) in the next Section for definition), under which the truncated Brownian signature path
up to degree $n$ is non-self-intersecting outside a set of
zero $(r,q)$-capacity. 

Intersection properties for random walks and stochastic processes is a classical topic in probability theory, 
and it has important applications in statistical field theory.
The non-self-intersection of sample paths of Brownian motion was studied extensively
in the literature. The first result dates back to 1944, in which Kakutani
\cite{Kakutani44} proved that almost every sample path of Brownian
motion is non-self-intersecting if the dimension $d\geqslant5$. Later
on, it was known by Dvoretzky, Erd\H{o}s and Kakutani \cite{DEK50}
that the optimal dimension is $d=4.$ The technique of Kakutani was
extended to the capacity setting on Wiener space by Fukushima \cite{Fukushima84}.
In particular, he showed that outside a set of zero $(1,2)$-capacity, every
sample path of Brownian motion is non-self-intersecting if $d\geqslant7.$ This
result was further extended by Takeda \cite{Takeda84} for general
$(r,q)$-capacities under the constraint $d>rq+4.$ It is remarkable
that in Fukushima's setting, Lyons \cite{Lyons86} proved that the
optimal dimension is $d=6$. However, it
is not known (and we expect that it is not true) that outside a slim set 
every sample path of Brownian motion is non-self-intersecting regardless of
the dimension $d$. 

Our technique of proving the quasi-sure non-self-intersection of the
Brownian signature path is inspired by the general ideas contained
in the aforementioned series of works. In particular, the key ingredient
is to establish a maximal type capacity estimate and a small ball
capacity estimate for the signature path. However, it will be clear
that our technique is robust enough to be extended to more general
Gaussian processes as it does not rely on the explicit distribution
of Brownian motion and any martingale properties, which is indeed
the case for the aforementioned works. In contrast to the Brownian
motion, as the Brownian signature path is an infinite dimensional
process taking values in the algebra of tensor series, it is not entirely
surprising that a quasi-sure non-self-intersection result can be expected.

According to Sugita's work in \cite{Sugita88}, our result implies
the corresponding almost-sure non-degeneracy property and uniqueness
of signature result for any probability measure on $W$ associated
with a positive generalized Wiener functional.

The present paper is organized in the following way. In Section 2
we formulate our main result, in Section 3 we develop the proofs
and in Section 4 we give a few remarks as conclusion.

\section{Formulation of Main Result}

In this section, we present the basic notions in quasi-sure analysis
and formulate our main result. We refer the reader to \cite{Malliavin97}
and \cite{Shigekawa98} for a systematic introduction to the Malliavin
calculus and quasi-sure analysis.

Let $(W,\mathcal{B}(W),\mathbb{P})$ be the canonical Wiener space
over $\mathbb{R}^{d}.$ In other words, $W$ is the space of continuous
paths $w:\ [0,1]\rightarrow\mathbb{R}^{d}$ starting at the origin
equipped with the uniform topology, $\mathcal{B}(W)$ is the Borel
$\sigma$-algebra and $\mathbb{P}$ is the canonical Wiener measure.
Let $\mathcal{H}$ be the space of absolutely continuous paths in
$W$ with square integrable derivative with respect to the Lebesgue
measure. It is well known that the canonical embedding $\iota:\ \mathcal{H}\rightarrow W$
give rise to the structure of an abstract Wiener space in the sense
of Gross. Let $\iota^{*}:\ W^{*}\rightarrow\mathcal{H}^{*}\cong\mathcal{H}$
be the corresponding dual embedding.

Consider the space $\mathcal{P}$ of polynomial functionals over $W,$
which consists of functionals of the form $F=f(\varphi_{1},\cdots,\varphi_{n})$,
where $f$ is a polynomial over $\mathbb{R}^{n}$ and $\varphi_{1},\cdots,\varphi_{n}\in W^{*}.$
The \textit{Malliavin derivative} of $F$ is the $\mathcal{H}$-valued
functional 
\[
DF=\sum_{i=1}^{n}\frac{\partial f}{\partial x^{i}}(\varphi_{1},\cdots,\varphi_{n})\iota^{*}\varphi_{i}.
\]
This definition extends to Hilbert space valued polynomial functionals
in a natural way. In particular, the $r$-th derivative of $F\in\mathcal{P}$
can be defined inductively as an $\mathcal{H}^{\otimes r}$-valued
functional. For $r\in\mathbb{N}$ and $q\geqslant1$, the $(r,q)$-Sobolev
norm of $F$ is defined to be 
\[
\|DF\|_{r,q}=\left(\sum_{i=0}^{r}\mathbb{E}[\|D^{i}F\|_{\mathcal{H}^{\otimes i}}^{q}]\right)^{\frac{1}{q}}.
\]
The $(r,q)$-\textit{Sobolev space} $\mathbb{D}_{r,q}$ is the completion
of $\mathcal{P}$ under the $(r,q)$-Sobolev norm. 

Throughout the rest we always assume that $r\in\mathbb{N}$ and $q>1.$ 

Let $O$ be an open subset of $W.$ The $(r,q)$\textit{-capacity}
of $O$ is defined to be 
\[
\mathrm{Cap}_{r,q}(O)=\inf\{\|F\|_{r,q}:\ F\in\mathbb{D}_{r,q},\ F\geqslant1\ \mathrm{on}\ O,\ F\geqslant0\ \mathrm{on}\ W\ \mathrm{for}\ \mathbb{P}\mathrm{-a.s.}\}.
\]
For a general subset $A\subset W,$ its $(r,q)$-capacity is defined
to be 
\begin{equation}
\mathrm{Cap}_{r,q}(A)=\inf\{\mathrm{Cap}_{r,q}(O):\ O\ \mathrm{open,\ }A\subset O\}.\label{eq: capacity}
\end{equation}
It is not hard to see that the $(r,q)$-capacity is non-negative,
increasing and sub-additive. Moreover, the following inequality holds:
\[
\mathbb{P}(A)^{\frac{1}{q}}=\mathrm{Cap}_{0,q}(A)\leqslant\mathrm{Cap}_{r,q}(A),\ \forall A\in\mathcal{B}(W).
\]
Therefore, capacities are finer scales in measuring the size of a
set from an analytic view point than the underlying probability measure. 

According to Malliavin, a \textit{slim} set is a subset having zero
$(r,q)$-capacity for every $(r,q).$ A property on paths is said
to hold \textit{quasi-surely} if it holds outside a slim set. We are
interested in properties which hold quasi-surely. 

We will also be working with functions that are defined quasi-surely.
A function $f$ on $W$ is said to be $(r,q)$-\textit{quasi-continuous}
if for every $\varepsilon>0,$ there exists an open subset $O\subset W$,
such that $\mathrm{Cap}_{r,q}(O)<\varepsilon$ and $f|_{O^{c}}$ is
continuous. A main property for quasi-continuous functions that we
will be using is a version of Chebyshev's inequality (see \cite{Malliavin97},
Theorem 2.2):
\begin{equation}
\mathrm{Cap}_{r,q}(|f|>R)\leqslant\frac{M_{r,q}\|f\|_{r,q}}{R},\ R>0,\label{eq: Chebyshev inequality}
\end{equation}
for any $(r,q)$-quasi-continuous function $f\in\mathbb{D}_{r,q},$ where
$M_{r,q}$ is a constant depending only on $r$ and $q.$

Now we are in a position to formulate our main result. 

The basic object we are interested in is the \textit{Brownian rough
path}
\[
\mathbf{w}=(1,w^{1},w^{2}):\ [0,1]\rightarrow G^{2}(\mathbb{R}^{d})=\exp\left(\mathbb{R}^{d}\oplus[\mathbb{R}^{d},\mathbb{R}^{d}]\right)
\]
in dimension $d\geqslant2$, which is the canonical lifting of Brownian motion to the free
nilpotent Lie group of order $2$ over $\mathbb{R}^d$. Heuristically, through the logarithmic diffeomorphism onto the Lie algebra, the Brownian rough path $\mathbf{w}$ is equivalent to the process 
$$\sum_{j=1}^{d}w_{t}^{j}e_{j}+\frac{1}{2}\sum_{1\leqslant j<k\leqslant d}\left(\int_{0}^{t}w_{s}^{j}dw_{s}^{k}-w_{s}^{k}dw_{s}^{j}\right)[e_{j},e_{k}],$$which is the original Brownian motion coupled with its Lévy area process.

In rough path theory, it is an important result of Lyons \cite{Lyons98}
that any rough path $\mathbf{X}$ with roughness $p$ (or a $p$-rough
path) extends uniquely to a continuous path $\mathbb{X}$ taking values
in the algebra 
\[
T((\mathbb{R}^{d}))=\mathbb{R}\oplus\mathbb{R}^{d}\oplus(\mathbb{R}^{d})^{\otimes2}\oplus\cdots
\]
of tensor series, such that the projection of $\mathbb{X}$ onto the
truncated tensor algebra up to every degree $n\geqslant\lfloor p\rfloor$
has finite $p$-variation. Here the truncated tensor algebra $T^{(n)}(\mathbb{R}^{d})$
is equipped with the Hilbert-Schmidt norm. Lyons' lifting $\mathbb{X}$
of a rough path $\mathbf{X}$ is also known as the \textit{signature
path} of $\mathbf{X}$. This is a generalized notion of taking iterated
path integrals in each degree, and the signature is just the end point
of the signature path (see (\ref{eq: signature})).

According to Inahama \cite{Inahama06} (see also \cite{BGQ14}), the Brownian rough path is quasi-surely well defined as the limit
of the lifting of dyadic piecewise linear interpolation of Brownian motion under $p$-variation metric for $2<p<3$. Therefore, from 
Lyons' extension theorem, the Brownian signature path and the signature are well defined quasi-surely.

\begin{rem}
Due to the multiplicative structure in $T((\mathbb{R}^{d})),$ a rough
path $\mathbf{X}$ can either mean an actual path $\mathbf{X}_{t}$
indexed by a single parameter $t$ or a multiplicative functional
$\mathbf{X}_{s,t}$ indexed by a pair $s\leqslant t.$ These two notions
are interchangeable with each other by setting $\mathbf{X}_{s,t}=\mathbf{X}_{s}^{-1}\otimes\mathbf{X}_{t}$
and $\mathbf{X}_{t}=\mathbf{X}_{0,t}.$ In this paper, except for
the situation in which we write down the notation explicitly with
a double subscript, when referring to a rough path or a signature
path, we always mean the actual path with one single parameter.
\end{rem}

Our main result can be stated as follows.

\begin{thm}
\label{thm: quantitative non-self-intersection}For $n\in\mathbb{N}$, define
$$\mathcal{O}_{n}=\{w\in W:\ S_{n}(\mathbf{w})_{0,s}=S_{n}(\mathbf{w})_{0,t}\ \mathrm{for\ some\ }0\leqslant s<t\leqslant1\},$$where
$S_n(\mathbf{w})_{0,t}$ is the truncated Brownian signature path up to degree $n$. Then $\mathcal{O}_{n}$
has zero $(r,q)$-capacity provided 
\[
\left(\begin{array}{c}
n+d-1\\
n
\end{array}\right)>rq+4.\label{eq: quantitative constraint}
\]
In particular, the Brownian signature path is non-self-intersecting quasi-surely.
\end{thm}

\begin{rem}
Takeda \cite{Takeda84} proved that if $d>rq+4$, 
outside a set of zero $(r,q)$-capacity every sample path of Brownian motion is non-self-intersecting. This corresponds to the case of $n=1$ in Theorem \ref{thm: quantitative non-self-intersection}. Therefore, our main result extends Takeda's result to the higher degree situation. 
\end{rem}

The non-self-intersection of the Brownian signature path has an important geometric interpretation on the Brownian rough path, which corresponds to the non-degeneracy property mentioned in the introduction and arises naturally from the uniqueness of signature problem in rough path theory. According to \cite{BGLY15}, this non-degeneracy property can be made precise by using the language of a real tree.
Recall that a \textit{real tree}
is a metric space $\tau$ in which every two distinct points can be
joined by a unique non-self-intersecting path (up to reparametrization),
and this path is a geodesic. 

A continuous path $x:\ [0,1]\rightarrow X$
in some topological space $X$ is called \textit{tree-like}, if there
exists a real tree $\tau,$ and two continuous maps $\alpha:\ [0,1]\rightarrow\tau$
and $\beta:\ \tau\rightarrow X$ such that $\alpha(0)=\alpha(1)$
and $x=\beta\circ\alpha.$ In other words, a tree-like path is a path
which can be realized as a loop in some real tree. A \textit{tree-like
piece} of a continuous path $x$ is a pair $s<t$ such that $x|_{[s,t]}$
is tree-like. 

From the feature of a real tree, it is clear that the aforementioned
non-degeneracy property means the fact that a path does not have any
tree-like pieces. 
\begin{defn}
A continuous path is called \textit{tree-reduced }if it does not have any
tree-like pieces.
\end{defn}

It is clear that if a path is non-self-intersecting, then it is tree-reduced.

According to the deterministic uniqueness result for signature in \cite{BGLY15}, we know that a weakly geometric
rough path (a continuous path in the free nilpotent group of order
$\lfloor p\rfloor$ with finite $p$-variation for some $p\geqslant1$)
is tree-like if and only if it has trivial signature. Therefore, a tree-like piece in a rough path corresponds to a loop in its signature path and vice versa. It follows immediately that Theorem \ref{thm: quantitative non-self-intersection} is \textit{equivalent} to the following, which is already interesting on its own. 

\begin{thm}
\label{thm: tree-reduced property}The Brownian rough path is tree-reduced quasi-surely.
\end{thm}

Another important consequence of Theorem \ref{thm: quantitative non-self-intersection} is a quasi-sure uniqueness result for the signature of Brownian motion. In the uniqueness of signature aspect, this partially extends the work of Le Jan and Qian \cite{LQ13} to the capacity setting. Note that their original work is stronger than uniqueness as it gives an explicit way to reconstruct a sample path of Brownian motion from its signature .

\begin{thm}
\label{thm: quasi-sure uniqueness of signature} Outside a slim set
$\mathcal{N}\subset W,$ two sample paths $w$ and $w'$ of Brownian motion have the
same signature if and only if they differ from each other by a reparametrization.
In other words, quasi-surely every sample path of Brownian motion is uniquely determined
by its signature up to reparametrization.
\end{thm}

\section{Proof of the Main Theorem}

In this section, we are going to develop the proof of Theorem \ref{thm: quantitative non-self-intersection} and point out how Theorem \ref{thm: quasi-sure uniqueness of signature} follows easily from this and the deterministic uniqueness result for signature.

Along the general ideas in the aforementioned works of Kakutani, Fukushima
and Takeda, our proof of Theorem \ref{thm: quantitative non-self-intersection}
contains three main steps: a large deviation type capacity estimate
for the maximal functional on the signature path, a small ball capacity
estimate for the signature path, and a subdivision argument.

A crucial point in our proof is a general and useful technique in
rough path theory on controlling higher degree signature components.
It consists of a quantitative statement of Lyons' lifting theorem
and a technique used by Hambly and Lyons \cite{HL98} in the construction
of stochastic area for Brownian motion on the Sierpinski gasket. We
state the result as follows and refer the reader to the monograph
by Lyons and Qian \cite{LQ02} for its proof. 
\begin{thm}
\label{thm: technique of controlling higher degree signature components}Let
$\mathbf{X}=(1,X^{1},\cdots,X^{\lfloor p\rfloor})$ be a $p$-rough
path.

(1) Suppose that $\mathbb{X}=(1,X^{1},X^{2},\cdots)$ is the signature
path of $\mathbf{X}$. If there exists a control function $\omega(s,t)$
such that 
\begin{equation}
\left|X_{s,t}^{i}\right|\leqslant\frac{\omega(s,t)^{\frac{i}{p}}}{\beta(i/p)!}\label{eq: Lyons' lifting}
\end{equation}
for $1\leqslant i\leqslant\lfloor p\rfloor$ and $0\leqslant s<t\leqslant1,$
where $\beta$ is a constant such that 
\begin{equation}
\beta\geqslant p^{2}\left(1+\sum_{l=1}^{\infty}\left(\frac{2}{l}\right)^{(\lfloor p\rfloor+1)/p}\right),\label{eq: beta}
\end{equation}
then the inequality (\ref{eq: Lyons' lifting}) holds for all $i>\lfloor p\rfloor$
as well. 

(2) Given a constant $\gamma>p-1,$ for $0\leqslant s\leqslant t\leqslant1$
and $1\leqslant i\leqslant\lfloor p\rfloor$, define 
\begin{equation}
\rho_{i}(\mathbf{X};s,t)=\sum_{m=1}^{\infty}m^{\gamma}\sum_{k=1}^{2^{m}}\left|X_{t_{m}^{k-1},t_{m}^{k}}^{i}\right|^{\frac{p}{i}},\label{eq: rho function}
\end{equation}
where $(t_{m}^{k})_{0\leqslant k\leqslant2^{m}}$ is the dyadic partition
of $[s,t].$ Then there exists a constant $C=C(p,\gamma)$, such that
\[
\sup_{\mathcal{P}([s,t])}\sum_{l}\left|X_{t_{l-1},t_{l}}^{i}\right|^{\frac{p}{i}}\leqslant C(p,\gamma)\sum_{j=1}^{i}\rho_{j}(\mathbf{X};s,t)
\]
for all $1\leqslant i\leqslant\lfloor p\rfloor$ and $0\leqslant s\leqslant t\leqslant1,$
where the supremum is taken over all finite partitions of $[s,t].$
\end{thm}

As we are interested in the Brownian rough path, throughout the rest,
we always fix $2<p<3$ and the two constants $\beta,\gamma$ arising
from Theorem \ref{thm: technique of controlling higher degree signature components}.
For simplicity, we always omit the dependence on $p,\beta$ and $\gamma$
for a constant, and the value of a constant may change from line to
line even the same notation is used.

\subsection{A Maximal Type Capacity Estimate}

As the first step, we are going to estimate the $(r,q)$-capacity
of the event $\{w:\max_{t\in[t_{0},t_{1}]}\left\Vert S_{n}(\mathbf{w})_{0,t}-S_{n}(\mathbf{w})_{0,t_{0}}\right\Vert >\eta\},$
where $[t_{0},t_{1}]$ is a dyadic sub-interval of $[0,1]$ (i.e.
$[t_{0},t_{1}]=[(k-1)/2^{m},k/2^{m}]$ for some $k,m$). Our main
idea is to control the maximal function by the series defined by (\ref{eq: rho function}),
and to observe that the increments $w_{s,t}^{i}$ ($i=1,2$) are ``evenly
distributed'' over a dyadic partition.

For $m\geqslant1,$ let $\mathbf{w}^{(m)}=(1,w^{(m),1},w^{(m),2})$
be the lifting of the dyadic piecewise linear interpolation of $w|_{[t_{0},t_{1}]}$
over the dyadic partition of $[t_{0},t_{1}]$ into small intervals
of length $1/2^{m}.$ In other words, $w^{(m),1}$ is just the increment
process of the interpolation, while $w^{(m),2}$ is defined by second
order iterated integrals of the interpolation.
\begin{lem}
\label{lem: L^2 estimates}We have the following estimates: 
\[
\sup_{m\geqslant1}\|w_{t_{0},t_{1}}^{(m),i}\|_{L^{2}}\leqslant C_{d}|t_{1}-t_{0}|^{\frac{i}{2}},\ \mathrm{for}\ i=1,2,
\]
and 
\[
\|w_{t_{0},t_{1}}^{(m+1),2}-w_{t_{0},t_{1}}^{(m),2}\|_{L^{2}}\leqslant C_{d}\frac{|t_{1}-t_{0}|}{2^{m/2}},\ \mathrm{for}\ m\geqslant1,
\]
where $C_{d}$ is a constant depending only on the dimension $d.$\end{lem}
\begin{proof}
See \cite{LQ02}, Proposition 4.4.1 (and also \cite{BGQ14}, Lemma
2.3).
\end{proof}

\begin{lem}
\label{lem: Sobolev estimate}Suppose $N\in\mathbb{N}.$ Then for
$i=1,2,$ we have $|w_{t_{0},t_{1}}^{i}|^{2N}\in\oplus_{j=0}^{2iN}\mathcal{H}_{j}$
and 
\[
\||w_{t_{0},t_{1}}^{i}|^{2N}\|_{4N,q}\leqslant C_{N,q,d}|t_{1}-t_{0}|^{iN},
\]
where $\mathcal{H}_{j}$ is the $j$-th Wiener-Itô chaos and $C_{N,q,d}$
is a constant depending only on $N,q$ and $d.$\end{lem}
\begin{proof}
We only need to consider the case when $i=2,$ as $w_{t_{0},t_{1}}^{1}$
is just the increment of Brownian motion in which case the assertion
is obvious.

First of all, we have 
\begin{eqnarray*}
\left|\left|w_{t_{0},t_{1}}^{(m+1),2}\right|^{2N}-\left|w_{t_{0},t_{1}}^{(m),2}\right|^{2N}\right| & \leqslant & \left|w_{t_{0},t_{1}}^{(m+1),2}-w_{t_{0},t_{1}}^{(m),2}\right|\\
 &  & \cdot\sum_{k=0}^{2N-1}\left|w_{t_{0},t_{1}}^{(m+1),2}\right|^{k}\left|w_{t_{0},t_{1}}^{(m),2}\right|^{2N-1-k}.
\end{eqnarray*}
From the hypercontractivity of the Ornstein-Uhlenbeck semigroup, it
is well known that the $L^{q}$ ($q>1$) and $L^{2}$-norms are comparable
over a given Wiener-Itô chaos. In particular, we have (see \cite{Shigekawa98},
Proposition 2.14): 
\begin{equation}
\|F\|_{L^{q}}\leqslant C_{N,q}\|F\|_{L^{2}}\label{eq: comparability of q and 2 norms}
\end{equation}
for any $F\in\oplus_{j=0}^{N}\mathcal{H}_{j}$. Since $|w_{t_{0},t_{1}}^{(m),2}|^{2k}\in\oplus_{j=0}^{4k}\mathcal{H}_{j},$
by using Lemma \ref{lem: L^2 estimates} and (\ref{eq: comparability of q and 2 norms}),
it is straight forward to see that 
\begin{equation}
\left\Vert \left|w_{t_{0},t_{1}}^{(m+1),2}\right|^{2N}-\left|w_{t_{0},t_{1}}^{(m),2}\right|^{2N}\right\Vert _{L^{2}}\leqslant C_{N,d}\frac{|t_{1}-t_{0}|^{2N}}{2^{m/2}}.\label{eq: L^2 norm of polynomial difference}
\end{equation}
Therefore, $\left|w_{t_{0},t_{1}}^{(m),2}\right|^{2N}$ converges
in $L^{2}$ as $m\rightarrow\infty.$ This shows that $|w_{t_{0},t_{1}}^{2}|^{2N}\in\oplus_{j=0}^{4N}\mathcal{H}_{j}$
as it is the almost-sure limit of $\left|w_{t_{0},t_{1}}^{(m),2}\right|^{2N}$. 

Moreover, the Sobolev norm of $\left|w_{t_{0},t_{1}}^{(m+1),2}\right|^{2N}-\left|w_{t_{0},t_{1}}^{(m),2}\right|^{2N}$
can by controlled by the $L^{2}$-norm uniformly as they are polynomials
of a fixed degree (see \cite{BGQ14}, Lemma 2.2) . In particular,
we obtain from (\ref{eq: L^2 norm of polynomial difference}) that
\[
\left\Vert \left|w_{t_{0},t_{1}}^{(m+1),2}\right|^{2N}-\left|w_{t_{0},t_{1}}^{(m),2}\right|^{2N}\right\Vert _{4N,q}\leqslant C_{N,q,d}\frac{|t_{1}-t_{0}|^{2N}}{2^{m/2}}.
\]
This implies that $w_{t_{0},t_{1}}^{(m),2}\rightarrow w_{t_{0},t_{1}}^{2}$
in $\mathbb{D}_{4N,q}$ as $m\rightarrow\infty.$

Therefore, by Lemma \ref{lem: L^2 estimates} we have 
\begin{eqnarray*}
\||w_{t_{0},t_{1}}^{2}|^{2N}\|_{4N,q} & = & \lim_{m\rightarrow\infty}\||w_{t_{0},t_{1}}^{(m),2}|^{2N}\|_{4N,q}\\
 & \leqslant & C_{N,q}\lim_{m\rightarrow\infty}\|w_{t_{0},t_{1}}^{(m),2}\|_{L^{2}}^{2N}\\
 & \leqslant & C_{N,q,d}|t_{1}-t_{0}|^{2N}.
\end{eqnarray*}

\end{proof}

Now we are able to establish the required maximal capacity estimate.
\begin{prop}
\label{prop: maximal estimates}(1) Let $\alpha_{p}=\sum_{j=1}^{\infty}1/(\beta(j/p)!)<1.$
Then for any $n\in\mathbb{N},$ we have:
\begin{eqnarray}
 &  & \mathrm{Cap}_{r,q}\left(\max_{t_{0}\leqslant t\leqslant t_{1}}\|S_{n}(\mathbf{w})_{t_{0},t}-\mathbf{1}\|>\alpha\right)\nonumber \\
 & \leqslant & \begin{cases}
C_{N,q,d}\frac{|t_{1}-t_{0}|^{N}}{\alpha^{2N}}\left(1+\frac{|t_{1}-t_{0}|^{N}}{\alpha^{2N}}\right), & \mathrm{if}\ 0<\alpha\leqslant\alpha_{p};\\
C_{N,q,d}\frac{|t_{1}-t_{0}|^{N}}{\alpha^{2N/n}}\left(1+\frac{|t_{1}-t_{0}|^{N}}{\alpha^{2N/n}}\right), & \mathrm{if}\ \alpha>\alpha_{p},
\end{cases}\label{eq: first maximal inequality}
\end{eqnarray}
where $N>r$ and $C_{N,q,d}$ is a constant depending only on $N,q,d$. 

(2) Suppose $N>r$ and $\delta>0.$ Then for any $n\in\mathbb{N}$
and $0<\eta<2^{-1/\delta}\wedge\alpha_{p}$, we have
\begin{align}
 & \mathrm{\ Cap}_{r,q}\left(\max_{t_{0}\leqslant t\leqslant t_{1}}\|S_{n}(\mathbf{w})_{0,t}-S_{n}(\mathbf{w})_{0,t_{0}}\|>\eta\right)\nonumber \\
\leqslant & \ C_{N,q,d}\left(\frac{|t_{1}-t_{0}|^{N}}{\eta^{2N(1+\delta)}}\cdot\left(1+\frac{|t_{1}-t_{0}|^{N}}{\eta^{2N(1+\delta)}}\right)+\eta^{2N\delta/n}\right),\label{eq: second maximal inequality}
\end{align}
where $C_{N,q,d}$ is a constant depending only on $N,q,d.$ \end{prop}
\begin{proof}
(1) Let $C_{p}=(1/p)!+(2/p)!$ and $\beta$ be given by (\ref{eq: beta}).
Define a control function $\omega(s,t)$ by 
\begin{equation}
\omega(s,t)=\beta C_{p}\sum_{i=1}^{2}\sup_{\mathcal{P}([s,t])}\sum_{l}|w_{t_{l-1},t_{l}}^{i}|^{\frac{p}{i}},\ 0\leqslant s\leqslant t\leqslant1.\label{eq: control function}
\end{equation}
Since $\mathbf{w}$ is a quasi-surely well defined $p$-rough path,
according to Theorem \ref{thm: technique of controlling higher degree signature components}
(1), we have 
\[
\max_{t_{0}\leqslant t\leqslant t_{1}}|w_{t_{0},t}^{i}|\leqslant\frac{\omega(t_{0},t_{1})^{i/p}}{\beta(i/p)!}
\]
for all $i\geqslant1,$ where $w^{i}$ denotes the $i$-th degree
component of the Brownian signature path. For the moment let $\lambda>0$
be such that 
\[
\sum_{i=1}^{n}\frac{\lambda^{i/p}}{\beta(i/p)!}\leqslant\alpha
\]
for given $\alpha>0.$ It follows that
\begin{align}
 & \left\{ w:\max_{t_{0}\leqslant t\leqslant t_{1}}\|S_{n}(\mathbf{w})_{t_{0},t}-\mathbf{1}\|>\alpha\right\} \nonumber \\
= & \left\{ w:\max_{t_{0}\leqslant t\leqslant t_{1}}\left(\sum_{i=1}^{n}|w_{t_{0},t}^{i}|^{2}\right)>\alpha^{2}\right\} \nonumber \\
\subseteq & \left\{ w:\sum_{i=1}^{n}\max_{t_{0}\leqslant t\leqslant t_{1}}|w_{t_{0},t}^{i}|^{2}>\sum_{i=1}^{n}\frac{\lambda^{2i/p}}{(\beta(i/p)!)^{2}}\right\} \nonumber \\
\subseteq & \bigcup_{i=1}^{n}\left\{ w:\max_{t_{0}\leqslant t\leqslant t_{1}}|w_{t_{0},t}^{i}|>\frac{\lambda^{i/p}}{\beta(i/p)!}\right\} \nonumber \\
\subseteq & \{w:\omega(t_{0},t_{1})>\lambda\}.\label{eq: controlling maximal by control}
\end{align}
Let $\rho_{i}(\mathbf{w};t_{0},t_{1})$ ($i=1,2$) be given by (\ref{eq: rho function}).
According to Theorem \ref{thm: technique of controlling higher degree signature components}
(2), we obtain that
\begin{eqnarray}
\{w:\omega(t_{0},t_{1})>\lambda\} & \subseteq & \left\{ w:\rho_{1}(\mathbf{w};t_{0},t_{1})>C\lambda\right\} \nonumber \\
 &  & \bigcup\left\{ w:\rho_{2}(\mathbf{w};t_{0},t_{1})>C\lambda\right\} ,\label{eq: controlling control by rho}
\end{eqnarray}
where $C>0$ is some constant depending only on $p$ and $\gamma$
in that theorem. 

For $\theta>0$, let $C_{\theta}>0$ be a constant such that 
\[
C_{\theta}\sum_{m=1}^{\infty}m^{\gamma}2^{-m\theta}\leqslant C.
\]
It follows that 
\begin{eqnarray*}
\left\{ w:\rho_{i}(\mathbf{w};t_{0},t_{1})>C\lambda\right\}  & \subseteq & \bigcup_{m=1}^{\infty}\left\{ w:\sum_{k=1}^{2^{m}}\left|w_{t_{m}^{k-1},t_{m}^{k}}^{i}\right|^{\frac{p}{i}}>C_{\theta}\lambda2^{-m\theta}\right\} \\
 & \subseteq & \bigcup_{m=1}^{\infty}\bigcup_{k=1}^{2^{m}}\left\{ w:\left|w_{t_{m}^{k-1},t_{m}^{k}}^{i}\right|^{\frac{p}{i}}>C_{\theta}\lambda2^{-m(\theta+1)}\right\} .
\end{eqnarray*}
Therefore, for any $N>r,$ we have 
\begin{eqnarray*}
 &  & \mathrm{Cap}_{r,q}\left(\rho_{i}(\mathbf{w};t_{0},t_{1})>C\lambda\right)\\
 & \leqslant & \sum_{m=1}^{\infty}\sum_{k=1}^{2^{m}}\mathrm{Cap}_{r,q}\left(\left|w_{t_{m}^{k-1},t_{m}^{k}}^{i}\right|^{\frac{p}{i}}>C_{\theta}\lambda2^{-m(\theta+1)}\right)\\
 & \leqslant & \sum_{m=1}^{\infty}\sum_{k=1}^{2^{m}}\mathrm{Cap}_{4N,q}\left(\left|w_{t_{m}^{k-1},t_{m}^{k}}^{i}\right|^{2N}>(C_{\theta}\lambda)^{\frac{2iN}{p}}2^{-\frac{2im(1+\theta)N}{p}}\right).
\end{eqnarray*}
On the other hand, from the proof of Lemma \ref{lem: Sobolev estimate},
we know that $|w_{t_{m}^{k-1},t_{m}^{k}}^{(l),i}|^{2N}\rightarrow|w_{t_{m}^{k-1},t_{m}^{k}}^{i}|^{2N}$
in $\mathbb{D}_{4N,q}$ as well as quasi-surely when $l\rightarrow\infty.$
Since $|w_{t_{m}^{k-1},t_{m}^{k}}^{(l),i}|^{2N}$ is continuous on
$W,$ according to \cite{Malliavin97}, Theorem 2.3.5, we see that
$|w_{t_{m}^{k-1},t_{m}^{k}}^{i}|^{2N}$ is $(4N,q)$-quasi-continuous.
Therefore, by using the Chebyshev inequality for capacity (\ref{eq: Chebyshev inequality})
and Lemma \ref{lem: Sobolev estimate}, we have
\begin{eqnarray*}
\mathrm{Cap}_{r,q}(\rho_{i}(\mathbf{w};t_{0},t_{1})>C\lambda) & \leqslant & C_{N,q}(C_{\theta}\lambda)^{-\frac{2iN}{p}}\sum_{m=1}^{\infty}2^{\frac{2im(1+\theta)N}{p}}\\
 &  & \cdot\sum_{k=1}^{2^{m}}\left\Vert \left|w_{t_{m}^{k-1},t_{m}^{k}}^{i}\right|^{2N}\right\Vert _{4N,q}\\
 & \leqslant & C_{N,q,d}(C_{\theta}\lambda)^{-\frac{2iN}{p}}|t_{1}-t_{0}|^{iN}\\
 &  & \cdot\sum_{m=1}^{\infty}2^{m\left(iN\left(\frac{2(1+\theta)}{p}-1\right)+1\right)}.
\end{eqnarray*}
Now we choose $\theta$ to be small enough such that 
\[
\left(\frac{2(1+\theta)}{p}-1\right)N+1<0.
\]
This is possible since $2<p<3$. Then we arrive at
\[
\mathrm{Cap}_{r,q}\left(\rho_{i}(\mathbf{w};t_{0},t_{1})>C\lambda\right)\leqslant C_{N,q,d}\frac{|t_{1}-t_{0}|^{iN}}{\lambda^{2iN/p}}.
\]

Combining with (\ref{eq: controlling maximal by control}) and (\ref{eq: controlling control by rho}),
we obtain that 
\begin{eqnarray*}
\mathrm{Cap}_{r,q}\left(\max_{t_{0}\leqslant t\leqslant t_{1}}\|S_{n}(\mathbf{w})_{t_{0},t}-\mathbf{1}\|>\alpha\right) & \leqslant & C_{N,q,d}\frac{|t_{1}-t_{0}|^{N}}{\lambda^{2N/p}}\\
 &  & \cdot\left(1+\frac{|t_{1}-t_{0}|^{N}}{\lambda^{2N/p}}\right).
\end{eqnarray*}
Now (\ref{eq: first maximal inequality}) follows by setting 
\[
\lambda=\begin{cases}
(\alpha/\alpha_{p})^{p}, & \mathrm{if}\ 0<\alpha\leqslant\alpha_{p};\\
(\alpha/\alpha_{p})^{p/n}, & \mathrm{if}\ \alpha>\alpha_{p}.
\end{cases}
\]

(2) From the multiplicative property of a rough path, we know that
\[
S_{n}(\mathbf{w})_{0,t}-S_{n}(\mathbf{w})_{0,t_{0}}=S_{n}(\mathbf{w})_{0,t_{0}}\otimes(S_{n}(\mathbf{w})_{t_{0},t}-\mathbf{1}).
\]
Therefore, for any $\delta>0,$ we have 
\begin{eqnarray*}
 &  & \{w:\max_{t_{0}\leqslant t\leqslant t_{1}}\|S_{n}(\mathbf{w})_{0,t}-S_{n}(\mathbf{w})_{0,t_{0}}\|>\eta\}\\
 & \subseteq & \{w:\|S_{n}(\mathbf{w})_{0,t_{0}}\|>\eta^{-\delta}\}\bigcup\{w:\max_{t_{0}\leqslant t\leqslant t_{1}}\|S_{n}(\mathbf{w})_{t_{0},t}-\mathbf{1}\|>\eta^{1+\delta}\}.
\end{eqnarray*}

If $0<\eta<\alpha_{p}$, then $\eta^{1+\delta}<\alpha_{p}$, and from
(\ref{eq: first maximal inequality}) we conclude that 
\begin{eqnarray*}
\mathrm{Cap}_{r,q}\left(\max_{t_{0}\leqslant t\leqslant t_{1}}\|S_{n}(\mathbf{w})_{t_{0},t}-\mathbf{1}\|>\eta^{1+\delta}\right) & \leqslant & C_{N,q,d}\frac{|t_{1}-t_{0}|^{N}}{\eta^{2N(1+\delta)}}\\
 &  & \cdot\left(1+\frac{|t_{1}-t_{0}|^{N}}{\eta^{2N(1+\delta)}}\right).
\end{eqnarray*}
On the other hand, as $\eta^{-\delta}/2>1>\alpha_{p},$ by applying
(\ref{eq: first maximal inequality}) for the case $[t_{0},t_{1}]=[0,1],$
we obtain that 
\begin{eqnarray*}
\mathrm{Cap}_{r,q}\left(\|S_{n}(\mathbf{w})_{0,t_{0}}\|>\eta^{-\delta}\right) & \leqslant & \mathrm{Cap}_{r,q}\left(\max_{0\leqslant t\leqslant1}\|S_{n}(\mathbf{w})_{0,t}-\mathbf{1}\|>\frac{1}{2}\eta^{-\delta}\right)\\
 & \leqslant & C_{N,q,d}\eta^{\frac{2N\delta}{n}}\left(1+\eta^{\frac{2N\delta}{n}}\right)\\
 & \leqslant & C_{N,q,d}\eta^{\frac{2N\delta}{n}}.
\end{eqnarray*}
Now (\ref{eq: second maximal inequality}) follows immediately.\end{proof}
\begin{rem}
In the probability measure case (i.e. $r=0$), it is possible to strengthen
the maximal inequalities in Proposition \ref{prop: maximal estimates}
to an exponential type by using a Fernique type estimate for the $p$-variation
of the Brownian rough path. However, this approach cannot be applied
to the capacity case as the Chebyshev inequality for capacity involves
the Sobolev norm instead of the $L^{q}$-norm. Indeed, it is even
not clear whether the $p$-variation is differentiable in the sense
of Malliavin.
\end{rem}

\subsection{A Small Ball Capacity Estimate}

The second step is to establish an estimate for the $(r,q)$-capacity
of the event $\{w:|S_{n}(\mathbf{w})_{0,t_{1}}-S_{n}(\mathbf{w})_{0,t_{0}}|\leqslant\eta\},$
where $t_{0},t_{1}$ are two dyadic points in $[0,1].$ The key ingredient
here is to observe the hypoellipticity of a collection of signature
components regarded as a stochastic differential equation (short for
SDE), so that the required estimate will follow from the Malliavin
calculus for hypoelliptic SDEs. It should be pointed out that $S_{n}(\mathbf{w})_{0,t}$,
as a path in the truncated tensor algebra, is not hypoelliptic as
it lives on the free nilpotent Lie group.

Recall from rough path theory (see \cite{FV10}, Proposition 7.8)
that the truncated signature path $S_{n}(\mathbf{w})_{t_{0},t}$ satisfies
the linear differential equation
\begin{equation}
\begin{cases}
dS_{n}(\mathbf{w})_{t_{0},t}=S_{n}(\mathbf{w})_{t_{0},t}\otimes dw_{t},\\
S_{n}(\mathbf{w})_{t_{0},t_{0}}=\mathbf{1}.
\end{cases}\label{eq: the differential equation for signature}
\end{equation}
In our case we can either interpret (\ref{eq: the differential equation for signature})
as a rough differential equation or a Stratonovich type SDE driven
by Brownian motion. Under the canonical basis of $\mathbb{R}^{d},$
we can write it as
\begin{equation}
dw_{t_{0},t}^{I}=w_{t_{0},t}^{I'}\circ dw^{i}\label{eq: SDE in coordinates}
\end{equation}
starting at zero, where $I$ runs over all words over $\{1,\cdots,d\}$
with length at most $n$, and $I'$ is the word obtained by dropping
the last letter $i$. We are interested in a \textit{consistent} collection
$\mathcal{I}$ of words in the sense that $I\in\mathcal{I}\implies I'\in\mathcal{I}$.
\begin{defn}
A word $I$ over $\{1,\cdots,d\}$ is said to be\textit{ non-degenerate
}if it has the form 
\[
I=\left(i_{0},\underset{l_{1}\;\text{copies}}{\underbrace{i_{1},\cdots,i_{1}}},\cdots,\underset{l_{k}\;\text{copies}}{\underbrace{i_{k},\cdots,i_{k}}}\right),
\]
where $1\leqslant i_{0}<i_{1}<\cdots<i_{k}\leqslant d$ and $l_{1},\cdots,l_{k}\geqslant0.$
\end{defn}

The collection $\mathcal{I}_{d,n}$ of non-degenerate words with length
at most $n$ is clearly consistent.
\begin{lem}
The cardinality of $\mathcal{I}_{d,n}$ is given by 
\[
|\mathcal{I}_{d,n}|=\left(\begin{array}{c}
n+d-1\\
n
\end{array}\right).
\]
\end{lem}
\begin{proof}
For $1\leqslant k\leqslant d,$ let $\mathcal{I}_{d,n}(k)$ be the set
of words $I\in\mathcal{I}_{d,n}$ whose first letter is $k.$ It is
not hard to see that there is a bijection between $\mathcal{I}_{d,n}(k)$
and the set of non-negative integer solutions to the equation
\[
x_{k+1}+\cdots+x_{d}+y=n-1.
\]
It follows that 
\[
|\mathcal{I}_{d,n}(k)|=\left(\begin{array}{c}
n-1+d-k\\
d-k
\end{array}\right).
\]
Therefore, 
\[
|\mathcal{I}_{d,n}|=\sum_{k=1}^{d}\left(\begin{array}{c}
n-1+d-k\\
d-k
\end{array}\right)=\sum_{l=n}^{n+d-1}\left(\begin{array}{c}
l-1\\
n-1
\end{array}\right).
\]
The last expression is easily seen to be $\left(\begin{array}{c}
n+d-1\\
n
\end{array}\right)$ as it can be modeled by choosing subsets of $\{1,\cdots,n+d-1\}$
with $n$ elements and with $l$ being the largest one.
\end{proof}

Now we have the following result.
\begin{lem}
\label{lem: hypoellipticity}The restriction of (\ref{eq: SDE in coordinates})
to the collection $\mathcal{I}_{d,n}$ of non-degenerate words defines
an $|\mathcal{I}_{d,n}|$-dimensional linear SDE satisfying Hörmander's
condition at the origin in the sense that the linear span of its generating
vector fields and their Lie brackets of any order at the origin is
$\mathbb{R}^{|\mathcal{I}_{d,n}|}$.\end{lem}
\begin{proof}
First of all, the consistency of $\mathcal{I}_{d,n}$ implies that the
restriction of (\ref{eq: SDE in coordinates}) to $\mathcal{I}_{d,n}$
is itself a linear SDE of dimension $|\mathcal{I}_{d,n}|.$ It suffices
to verify Hörmander's condition at the origin. 

We use the notation $X^{k;I}$ for a component to keep track of the
length $k$ of the word $I\in\mathcal{I}_{d,n}.$ In geometric notation,
the generating vector fields of the SDE (\ref{eq: SDE in coordinates})
restricted to $\mathcal{I}_{d,n}$ are given by
\[
V_{i}=\sum_{I_{i}}x^{k-1;I_{i}'}\partial_{k;I_{i}},\ 1\leqslant i\leqslant d,
\]
where we set $x^{0;I_{i}'}=1.$ Here the sum is taken over all words
$I_{i}\in\mathcal{I}_{d,n}$ whose last letter is $i,$ and $I'_{i}$
is the word obtained by dropping the last letter from $I_{i}.$ 

We write 
\[
V_{i}=\partial_{1;(i)}+P_{i},
\]
where 
\[
P_{i}=\sum_{|I_{i}|\geqslant2}x^{k-1;I_{i}'}\partial_{k;I_{i}}
\]
is a vector field with homogeneous linear coefficients. For $1\leqslant i<j\leqslant d,$
we then have 
\begin{eqnarray*}
[V_{i},V_{j}] & = & [\partial_{1;(i)}+P_{i},\partial_{1;(j)}+P_{j}]\\
 & = & \partial_{1;(i)}P_{j}-\partial_{1;(j)}P_{i}+[P_{i},P_{j}]\\
 & = & \partial_{2;(i,j)}-\partial_{1;(j)}P_{i}+[P_{i},P_{j}],
\end{eqnarray*}
where the $\partial P$ denotes the vector field obtained by differentiating
the coefficients of $P.$ 

Now the key observation is that if $i\leqslant j,$ then $P_{i}$
does not depend on $x^{1;(j)}.$ Indeed, if this is not the case,
then $I_{i}=(j,i)$ has to be a word appearing in the summation, contradicting
the construction of $\mathcal{I}_{d,n}.$ Therefore, 
\[
\partial_{1;(j)}P_{i}=0
\]
and we have 
\[
[V_{i},V_{j}]=\partial_{2;(i,j)}+[P_{i},P_{j}].
\]
Note that $[P_{i},P_{j}]$ is a vector field with homogeneous linear
coefficients of the form $x^{k-1;I_{i}'}$ or $x^{l-1;I_{j}'}$ ($k,l\geqslant2$).

If $1\leqslant i<j_{1}\leqslant j_{2}\leqslant d$, then 
\begin{eqnarray*}
[[V_{i},V_{j_{1}}],V_{j_{2}}] & = & [\partial_{2;(i,j_{1})}+[P_{i},P_{j_{1}}],\partial_{1;(j_{2})}+P_{j_{2}}]\\
 & = & \partial_{3;(i,j_{1},j_{2})}-\partial_{1;j_{2}}[P_{i},P_{j_{1}}]+[[P_{i},P_{j_{1}}],P_{j_{2}}].
\end{eqnarray*}
Again we know that the second term on the right hand side vanishes
as $[P_{i},P_{j_{1}}]$ does not depend on $x^{1;(j_{2})}.$ 

By an induction argument, we obtain the fact that 
\[
[\cdots[[V_{i},V_{j_{1}}],V_{j_{2}}],\cdots,V_{j_{m}}]=\partial_{m+1;(i,j_{1},\cdots,j_{m})}+P_{i,j_{1},\cdots,j_{m}},
\]
for all $1\leqslant i<j_{1}\leqslant\cdots\leqslant j_{m}\leqslant d$
and $1\leqslant m\leqslant n-1,$ where 
\[
P_{i,j_{1},\cdots,j_{m}}=[\cdots[[P_{i},P_{j_{1}}],P_{j_{2}}],\cdots,P_{j_{m}}]
\]
is a vector field with homogeneous linear coefficients not depending
on any $x^{1;(j)}$ with $j\geqslant j_{m}.$ In particular,
\[
P_{i,j_{1},\cdots,j_{m}}(0)=0
\]
and we conclude that the linear span of $V_{i}$ and their Lie brackets
at the origin coincides with $\mathrm{Span}\{\partial_{k;I}:I\in\mathcal{I}_{d,n}\}$,
which is $\mathbb{R}^{|\mathcal{I}_{d,n}|}.$
\end{proof}

According to Lemma \ref{lem: hypoellipticity} and Hörmander's theorem
from the Malliavin calculus, we know that the law of $(w_{t_{0},t}^{I})_{I\in\mathcal{I}_{d,n}}$
has a smooth density with respect to the Lebesgue measure. A small
ball probability estimate follows immediately from this fact. To obtain
a corresponding capacity estimate, we need to the following lemma.
\begin{lem}
\label{lem: Sobolev estimate for higher degrees}For any $i\geqslant1$,
we have $w_{t_{0},t_{1}}^{i}\in\oplus_{j=0}^{i}\mathcal{H}_{j}$ and
\begin{equation}
\|w_{t_{0},t_{1}}^{i}\|{}_{r,q}\leqslant C_{r,q,i,d}|t_{1}-t_{0}|^{\frac{i}{2}},\label{eq: Sobolev estimate for higher degrees}
\end{equation}
where $C_{r,q,i,d}$ is a constant depending only on $r,q,i$ and
$d.$\end{lem}
\begin{proof}
Let $\widetilde{\mathbf{w}}^{(m)}$ be the lifting of the $m$-th
dyadic piecewise linear interpolation of $w$ over $[0,1]$. Define
the control function $\omega_{m}(s,t)$ in the same way as in (\ref{eq: control function})
by replacing $\mathbf{w}$ by $\widetilde{\mathbf{w}}^{(m)}$. According
to Theorem \ref{thm: technique of controlling higher degree signature components},
we have 
\[
|\widetilde{w}_{s,t}^{(m),i}|\leqslant\frac{\omega_{m}(s,t)^{i/p}}{\beta(i/p)!},\ \forall i\geqslant1,
\]
and 
\[
\omega_{m}(s,t)\leqslant C\left(\rho_{1}(\widetilde{\mathbf{w}}^{(m)};s,t)+\rho_{2}(\widetilde{\mathbf{w}}^{(m)};s,t)\right),
\]
where $C$ is a constant depending only on $p$ and $\gamma.$ Therefore,
\[
\|w_{t_{0},t_{1}}^{(m),i}\|_{L^{2}}\leqslant C_{i}\left(\|\rho_{1}(\widetilde{\mathbf{w}}^{(m)};t_{0},t_{1})^{\frac{i}{p}}\|_{L^{2}}+\|\rho_{2}(\widetilde{\mathbf{w}}^{(m)};t_{0},t_{1})^{\frac{i}{p}}\|_{L^{2}}\right).
\]
It follows that 
\begin{eqnarray*}
\|\rho_{j}(\widetilde{\mathbf{w}}^{(m)};t_{0},t_{1})\|_{L^{\frac{2i}{p}}} & \leqslant & \sum_{l=1}^{\infty}l^{\gamma}\sum_{k=1}^{2^{l}}\||\widetilde{w}_{t_{l}^{k-1},t_{l}^{k}}^{(m),j}|^{\frac{p}{j}}\|_{L^{\frac{2i}{p}}}\\
 & \leqslant & C_{i}\sum_{l=1}^{\infty}l^{\gamma}\sum_{k=1}^{2^{l}}\|\widetilde{w}_{t_{l}^{k-1},t_{l}^{k}}^{(m),j}\|_{L^{2}}^{\frac{p}{j}}\\
 & \leqslant & C_{i,d}|t_{1}-t_{0}|^{\frac{p}{2}},
\end{eqnarray*}
where the last inequality follows from the fact that 
\[
\|\widetilde{w}_{s,t}^{(m),j}\|_{L^{2}}\leqslant C_{d}|t-s|^{\frac{j}{2}}
\]
for all $0\leqslant s\leqslant t\leqslant1$ and $m\geqslant1,$ even
in the case when $[s,t]$ is not a dyadic sub-interval of $[0,1]$.
This can be seen easily based on the computation in \cite{LQ02},
pp. 68--70. Therefore, we obtain that 
\begin{equation}
\|\widetilde{w}_{t_{0},t_{1}}^{(m),i}\|_{L^{2}}\leqslant C_{i,d}|t_{1}-t_{0}|^{\frac{i}{2}}.\label{eq: L^2 estimate for higher degrees}
\end{equation}

On the other hand, since $\widetilde{w}_{t_{0},t_{1}}^{(m),i}\in\oplus_{j=0}^{i}\mathcal{H}_{j},$
from (\ref{eq: L^2 estimate for higher degrees}) and (\ref{eq: comparability of q and 2 norms})
we know that the $L^{q}$-norm of $\widetilde{w}_{t_{0},t_{1}}^{(m),i}$
is uniformly bounded for any $q>2$. As $\widetilde{w}_{t_{0},t_{1}}^{(m),i}\rightarrow w_{t_{0},t_{1}}^{i}$
$\mathbb{P}$-almost-surely, it follows that the convergence holds
in $L^{2}$ as well. Therefore, $w_{t_{0},t_{1}}^{i}\in\oplus_{j=0}^{i}\mathcal{H}_{j}$
and it also satisfies (\ref{eq: L^2 estimate for higher degrees}).
Finally, (\ref{eq: Sobolev estimate for higher degrees}) follows
from \cite{BGQ14}, Lemma 2.2.
\end{proof}

Now we are able to establish the required small ball capacity estimate.
\begin{prop}
\label{prop: small ball capacity estimate}Given any $\tau>1,$ we
have the following estimate:

\[
\mathrm{Cap}_{r,q}\left(\left\Vert S_{n}(\mathbf{w})_{0,t_{1}}-S_{n}(\mathbf{w})_{0,t_{0}}\right\Vert \leqslant\eta\right)\leqslant\frac{C_{r,q,n,d,\tau}}{|t_{1}-t_{0}|^{\lambda_{n,d}/\tau q}}\eta^{\frac{|\mathcal{I}_{d,n}|}{\tau q}-r}
\]
for every $0<\eta<1$ and $n\in\mathbb{N},$ where $C_{r,q,n,d,\tau}$
is a constant depending only on $r,q,n,d$ and $\tau.$\end{prop}
\begin{proof}
Write $X_{t_{0},t}=(w_{t_{0},t}^{I})_{I\in\mathcal{I}_{d,n}}$ as a
diffusion in $\mathbb{R}^{|\mathcal{I}_{d,n}|}.$ By the multiplicative
structure of the signature path and the fact that $\|S_{n}(\mathbf{w})_{0,t_{0}}\|\geqslant1,$
we have 
\begin{eqnarray*}
\mathrm{Cap}_{r,q}\left(\|S_{n}(\mathbf{w})_{0,t_{1}}-S_{n}(\mathbf{w})_{0,t_{0}}\|\leqslant\eta\right) & \leqslant & \mathrm{Cap}_{r,q}\left(\|S_{n}(\mathbf{w})_{t_{0},t_{1}}-\mathbf{1}\|\leqslant\eta\right)\\
 & \leqslant & \mathrm{Cap}_{r,q}\left(\|X_{t_{0},t_{1}}\|\leqslant\eta\right).
\end{eqnarray*}

Now consider a function $f\in C^{\infty}(\mathbb{R}^{|\mathcal{I}_{d,n}|})$
such that 
\[
\begin{cases}
0\leqslant f\leqslant1,\\
f=1\ \mathrm{on}\ |x|\leqslant\eta\ \mathrm{and}\ f=0\ \mathrm{on}\ |x|\geqslant2\eta,\\
\left|\nabla^{k}f\right|\leqslant\frac{C_{r}}{\eta^{k}},\ \mathrm{for}\ k\leqslant r,
\end{cases}
\]
where $C_{r}$ is a constant depending only on $r.$ Let $F=f(X_{t_{0},t_{1}}).$
It follows that is smooth in the sense of Malliavin. By using Lemma
\ref{lem: Sobolev estimate for higher degrees} and the chain rule,
we obtain that 
\begin{equation}
\|F\|_{r,q'}\leqslant\frac{C_{r,q',n,d}}{\eta^{r}},\ \forall q'>1.\label{eq: Sobolev estimate for test function}
\end{equation}
Moreover, the same reason as in the proof of Proposition \ref{prop: maximal estimates}
shows that $F$ is $(r,q)$-quasi-continuous. Therefore, according
to the Chebyshev inequality (\ref{eq: Chebyshev inequality}), for
any $\tau>1,$ we have
\begin{eqnarray*}
\mathrm{Cap}_{r,q}(\|X_{t_{0},t_{1}}\|\leqslant\eta) & \leqslant & \mathrm{Cap}_{r,q}(F\geqslant1)\\
 & \leqslant & C_{r,q}\|F\|_{r,q}\\
 & \leqslant & C_{r,q}\sum_{i=0}^{r}\left(\mathbb{E}[\|D^{i}F\|^{q}\mathbf{1}_{\{|X_{t_{0},t_{1}}|\leqslant2\eta\}}]\right)^{\frac{1}{q}}\\
 & \leqslant & C_{r,q}\|F\|_{r,q_{1}}\mathbb{P}(\|X_{t_{0},t_{1}}\|\leqslant2\eta)^{\frac{1}{\tau q}}\\
 & \leqslant & \frac{C_{r,q,n,d,\tau}}{\eta^{r}}\mathbb{P}(\|X_{t_{0},t_{1}}\|\leqslant2\eta)^{\frac{1}{\tau q}},
\end{eqnarray*}
where $q_{1}=\tau q/(\tau-1).$

Finally, according to Lemma \ref{lem: hypoellipticity} and Hörmander's
theorem from the Malliavin calculus (here we use a quantitative version
in \cite{Shigekawa98}, Theorem 6.16), $X_{t_{0},t_{1}}$ has a smooth
density $p_{t_{0},t_{1}}(x)$ with respect to the Lebesgue measure
on $\mathbb{R}^{|\mathcal{I}_{d,n}|}$. In particular, $p_{t_{0},t_{1}}(x)$
satisfies the following estimate:
\[
\sup_{x\in\mathbb{R}^{|\mathcal{I}_{d,n}|}}p_{t_{0},t_{1}}(x)\leqslant\frac{C_{n,d}}{(t_{1}-t_{0})^{\lambda_{n,d}}}.
\]
Therefore, 
\begin{eqnarray*}
\mathrm{Cap}_{r,q}\left(\|X_{t_{0},t_{1}}\|\leqslant\eta\right) & \leqslant & \frac{C_{r,q,n,d,\tau}}{\eta^{r}}\left(\int_{\{x:|x|\leqslant2\eta\}}p_{t_{0},t_{1}}(x)dx\right)^{\frac{1}{\tau q}}\\
 & \leqslant & \frac{C_{r,q,n,d,\tau}}{|t_{1}-t_{0}|^{\lambda_{n,d}/\tau q}}\eta^{\frac{|\mathcal{I}_{d,n}|}{\tau q}-r}.
\end{eqnarray*}

\end{proof}

\subsection{Kakutani's Sub-division Argument}

Following the original sub-division argument of Kakutani \cite{Kakutani44},
we are now in a position to complete the proof of Theorem \ref{thm: quantitative non-self-intersection}. 

Here a remarkable fact that we are going to use is the sub-additivity
for the $q$-th power of the $(r,q)$-capacity (up to a constant depending
only on $r,q$) instead of the original sub-additivity. This sub-additivity
can be easily seen by using the renowned Meyer's inequalities and
the integral representation formula (see \cite{Shigekawa98}, Chapter 4) 
\[
(I-\mathcal{L})^{-\frac{r}{2}}=\frac{1}{\Gamma(r/2)}\int_{0}^{\infty}e^{-t}t^{\frac{r}{2}-1}T_{t}dt,
\]
where $\mathcal{L}$ is the generator of the Ornstein-Uhlenbeck semigroup
$T_{t}.$ If we use the sub-additivity for the capacity itself, we
will end up with the quantitative constraint $|\mathcal{I}_{d,n}|>rq+4q$
for the non-self-intersection property, which is not as sharp
as the version we are going to obtain.

\begin{proof}[Proof of Theorem \ref{thm: quantitative non-self-intersection}]

Fix two dyadic sub-intervals $[s_{0},s_{1}]$, $[t_{0},t_{1}]$ with
equal length $\Delta$ and $s_{1}<t_{0}.$ Let $\mathcal{A}_{n}$
be the event that $S_{n}(\mathbf{w})_{0,s}=S_{n}(\mathbf{w})_{0,t}$
for some $s\in[s_{0},s_{1}]$ and $t\in[t_{0},t_{1}].$ It follows
that 
\begin{eqnarray*}
\mathcal{A}_{n} & \subseteq & \left\{ w:\|S_{n}(\mathbf{w})_{0,s_{0}}-S_{n}(\mathbf{w})_{0,t_{0}}\|\leqslant2\eta\right\} \\
 &  & \bigcup\{w:\max_{s_{0}\leqslant s\leqslant s_{1}}\|S_{n}(\mathbf{w})_{0,s}-S_{n}(\mathbf{w})_{0,s_{0}}\|>\eta\}\\
 &  & \bigcup\{w:\max_{t_{0}\leqslant t\leqslant t_{1}}\|S_{n}(\mathbf{w})_{0,t}-S_{n}(\mathbf{w})_{0,t_{0}}\|>\eta\}
\end{eqnarray*}
for every $\eta>0.$ Combining with Proposition \ref{prop: maximal estimates}
and Proposition \ref{prop: small ball capacity estimate}, we have
\begin{eqnarray}
\mathrm{Cap}_{r,q}(\mathcal{A}_{n})^{q} & \leqslant & \frac{C_{r,q,n,d,\tau}}{|t_{0}-s_{0}|^{\lambda_{n,d}/\tau}}\eta^{\frac{|\mathcal{I}_{d,n}|}{\tau}-rq}\nonumber \\
 &  & +C_{N,q,d}\left(\frac{\Delta^{qN}}{\eta^{2Nq(1+\delta)}}\left(1+\frac{\Delta^{qN}}{\eta^{2Nq(1+\delta)}}\right)\right.\nonumber \\
 &  & \left.+\eta^{\frac{2Nq\delta}{n}}\right),\label{eq: before sub-division}
\end{eqnarray}
for any $\tau>1,N>r,\delta>0$ and small $\eta.$

Now we divide the intervals $[s_{0},s_{1}]$ and $[t_{0},t_{1}]$
into dyadic sub-intervals with length $\Delta/2^{l}.$ Note that any
$I\subseteq[s_{0},s_{1}]$ and $J\subseteq[t_{0},t_{1}]$ are separated
from each other by distance at least $t_{0}-s_{1}.$ Therefore, by
applying (\ref{eq: before sub-division}) to the dyadic sub-intervals,
we obtain that 
\begin{eqnarray*}
\mathrm{Cap}_{r,q}(\mathcal{A}_{n})^{q} & \leqslant & \frac{C_{r,q,n,d,\tau}}{|t_{0}-s_{1}|^{\lambda_{n,d}/\tau q}}2^{2l}\eta^{\frac{|\mathcal{I}_{d,n}|}{\tau}-rq}\\
 &  & +C_{N,q,d}2^{2l}\left(\frac{2^{-Nql}\Delta^{Nq}}{\eta^{2Nq(1+\delta)}}\left(1+\frac{2^{-Nql}\Delta^{Nq}}{\eta^{2Nq(1+\delta)}}\right)\right.\\
 &  & \left.+\eta^{\frac{2Nq\delta}{n}}\right).
\end{eqnarray*}
Setting $\eta=2^{-\sigma l}$ with $\sigma>0,$ we arrive at 
\begin{eqnarray}
\mathrm{Cap}_{r,q}(\mathcal{A}_{n})^{q} & \leqslant & \frac{C_{r,q,n,d,\tau}}{|t_{0}-s_{1}|^{\lambda_{n,d}/\tau}}2^{-l\left(\sigma\left(\frac{|\mathcal{I}_{d,n}|}{\tau}-rq\right)-2\right)}\nonumber \\
 &  & +C_{N,q,d}2^{-l\left(Nq(1-2\sigma(1+\delta))-2\right)}\nonumber \\
 &  & \cdot\left(1+2^{-Nql(1-2\sigma(1+\delta))}\right).\label{eq: before taking limit}
\end{eqnarray}

To expect that the right hand side goes to zero as $l\rightarrow\infty,$
we choose the parameters $\tau,\delta$ and $\sigma$ such that 
\begin{eqnarray*}
\sigma\left(\frac{|\mathcal{I}_{d,n}|}{\tau}-rq\right)>2
\end{eqnarray*}
and
\[
1-2\sigma(1+\delta)>0.
\]
This is equivalent to 
\begin{equation}
\frac{2}{|\mathcal{I}_{d,n}|/\tau-rq}<\sigma<\frac{1}{2(1+\delta)},\label{eq: constraint}
\end{equation}
provided that the left hand side is positive. As $\tau>1$ and $\delta>0$
is arbitrary, when 
\[
|\mathcal{I}_{d,n}|>rq+4,
\]
a choice of parameters satisfying (\ref{eq: constraint}) is certainly
possible. Therefore, after choosing $\tau,\delta$ and $\sigma,$
when $N$ is large we conclude that the right hand side of (\ref{eq: before taking limit})
converges to zero as $l\rightarrow\infty.$ In other words, we have
\[
\mathrm{Cap}_{r,q}(\mathcal{A}_{n})=0.
\]

On the other hand, if $S_{n}(\mathbf{w})_{0,s}=S_{n}(\mathbf{w})_{0,t}$
for some $s<t,$ apparently there exist two disjoint dyadic sub-intervals
$[s_{0},s_{1}]$ and $[t_{0},t_{1}]$ containing $s$ and $t$ respectively.
Therefore, 
\[
\mathrm{Cap}_{r,q}(\mathcal{O}_{n})=0,
\]
which gives the first assertion of Theorem \ref{thm: quantitative non-self-intersection}.

To prove the second assertion, let $\mathcal{O}$ be the event that the Brownian signature path has self-intersection at some $0\leqslant s<t\leqslant1$. Then we have $$\mathcal{O}\subseteq\bigcap_{n\geqslant1}\mathcal{O}_n.$$ Therefore, $\mathcal{O}$ has zero $(r,q)$-capacity for every $r$ and $q$. In other words, $\mathcal{O}$ is a slim set. 
\end{proof}

Based on the deterministic uniqueness result for signature in \cite{BGLY15}, it is not hard to see that the quasi-sure uniqueness for the signature of Brownian motion (Theorem \ref{thm: quasi-sure uniqueness of signature}) is a direct consequence of Theorem \ref{thm: quantitative non-self-intersection}.

\begin{proof}[Proof of Theorem \ref{thm: quasi-sure uniqueness of signature}]

From \cite{BGLY15}, Theorem 4.1, we know that the space of
signatures for weakly geometric $p$-rough paths has a canonical real
tree structure. In particular, if $g$ is the signature of some weakly
geometric $p$-rough path $\mathbf{X},$ then there exists a unique
weakly geometric $p$-rough path $\widetilde{\mathbf{X}}$ (up to
reparametrization) such that its signature is $g$ and its signature
path is non-self-intersecting.

In our case, let $\mathcal{N}$ be the slim set outside which every
Brownian signature path is non-self-intersecting. Suppose $w,w'\in\mathcal{N}^{c}$
are two sample paths of Brownian motion with the same signature. It follows that the
corresponding rough paths $\mathbf{w}$ and $\mathbf{w}'$
differ by a reparametrization, and hence $w$ are $w'$ differ by
a reparametrization. Therefore, quasi-surely every sample path of Brownian motion is
uniquely determined by its signature up to reparametrization. 

\end{proof}

\section{Final Remarks}

We give a few remarks to conclude our paper.

First of all, from the details of the proof, it is not hard to see
that our technique is robust as it only involves the Gaussian nature
of Brownian motion and the structure of its covariance function. In
particular, its explicit distribution, martingale property and Markov
property are not used at all. Therefore, 
our work extends to any Gaussian rough path under the intrinsic capacities induced by the underlying Gaussian measure
over the associated abstract Wiener space, for the cases where the Gaussian rough path $\mathbf{X}$ is well defined quasi-surely
and Hörmander's theorem for
rough differential equations driven by $\mathbf{X}$ is applicable. A fundamental example where everything works is the fractional Brownian
motion $B^{H}$ with Hurst parameter $H>1/4$. 

On the other hand, one might ask if we could strengthen Theorem \ref{thm: quantitative non-self-intersection}
to the intrinsic dimension of the truncated signature path $S_{n}(\mathbf{w})_{0,t}$
instead of restricting it to the collection $\mathcal{I}_{d,n}$ of components.
Indeed, it is known that (see \cite{FV10}) $S_{n}(\mathbf{w})_{0,t}$
satisfies an intrinsic hypoelliptic differential equation 
\[
dS_{n}(\mathbf{w})_{0,t}=\sum_{i=1}^d U_{i}(S_{n}(\mathbf{w})_{0,t})\circ dw_{t}^{i}
\]
on the free nilpotent Lie group $G^{n}(\mathbb{R}^{d})$ of order
$n$ over $\mathbb{R}^{d}$. Therefore, it is reasonable to expect
that $S_{n}(\mathbf{w})_{0,t}$ is non-self-intersecting outside a
set of zero $(r,q)$-capacity provided 
\[
\dim G^{n}(\mathbb{R}^{d})>rq+4.
\]
This will be sharper as $|\mathcal{I}_{d,n}|$ grows with rate $n^{d}$
while $\dim G^{n}(\mathbb{R}^{d})$ grows with rate $d^{n}/n$ as
$n\rightarrow\infty.$ However, what is missing is the analysis on
the vector fields $U_{i}$ in order to guarantee a priori estimates
on the density which is needed in our proof. This is non-trivial as
the vector fields are in fact polynomial of degree $n$ when pulled
back to the free nilpotent Lie algebra. It is not clear how to develop
a localization method which is consistent with our argument. We do
not pursue this direction because unlike the full signature
path, the truncated signature path up to a given degree does not have a
natural interpretation on the geometric behavior of the Brownian rough
path.

However, the case when $n=2$ is particularly interesting because
it is just the Brownian rough path. In this case, 
\[
\dim G^{2}(\mathbb{R}^{d})=|\mathcal{I}_{d,2}|=\frac{d^{2}+d}{2}.
\]
Based on the works of Dvoretzky, Erd\H{o}s and Kakutani \cite{DEK50} and Lyons \cite{Lyons86}
as we mentioned in the introduction, it is natural to expect that
the Brownian rough path has self-intersection with positive probability
when $d=2$ while it is non-self-intersecting outside a set of zero
$(1,2)$-capacity when $d=3.$ Moreover, it is even not unreasonable to 
expect that the Brownian rough path is non-self-intersecting outside a set of zero
$(r,2)$-capacity if and only if
$d^2+d\geqslant4r+8$.

\end{document}